\documentclass[11pt]{amsart} 
\usepackage{amsmath,amssymb,a4wide} 
\usepackage{mathabx}
\usepackage{color}
\usepackage{subcaption}
\usepackage{mathrsfs}
\usepackage{tikz,pgfplots}
\newtheorem{theorem}{Theorem} 
\newtheorem{lemma}[theorem]{Lemma} 
\newtheorem{proposition}[theorem]{Proposition}

\theoremstyle{definition}
\newtheorem{remark}[theorem]{Remark}

\newcommand{\R}{{\mathbb R}} 
\newcommand{\IR}{{\mathbb R}} 
\newcommand{\IE}{{\mathbb E}} 
\newcommand{\ts}{{\thinspace}} 
 
\newcommand{\N}{{\mathbb N}}

\newcommand{\E}{{\mathbb E}}

\newcommand{\dd}{{\rm d}}
\newcommand{\bigo}[1]{\mathcal{O}(#1)}

\providecommand{\norm}[1]{\left\lVert#1\right\rVert}

\usepackage{ulem}
\DeclareMathOperator{\Tr}{Tr}

\title[]
{Drift-preserving numerical integrators for stochastic 
{P}oisson systems}

\date{\today}

\author{David Cohen}
              \address{Department of Mathematics and Mathematical
              Statistics, Ume{\aa} University, 90187~Ume{\aa}, 
              Sweden} 
              \email{\tt david.cohen@umu.se} 
              \address{Department of Mathematical Sciences, 
              Chalmers University of Technology and University of Gothenburg, 41296~Gothenburg, Sweden}
              \email{\tt david.cohen@chalmers.se}
              
\author{Gilles Vilmart}
       \address{Section de math\'ematiques, 
        Universit\'e de Gen\`eve, Uni Dufour, 24 rue du Général Dufour,
Case postale 64, 1211~Geneva~4, Switzerland} 
        \email{\tt Gilles.Vilmart@unige.ch}

\begin{document}


\begin{abstract}
We perform a numerical analysis of a class of randomly perturbed {H}amiltonian systems and {P}oisson systems. 
For the considered additive noise perturbation of such systems, we show the long time behavior 
of the energy and quadratic Casimirs for the exact solution. 
We then propose and analyze a drift-preserving splitting scheme for such problems with the following 
properties: exact drift preservation of energy and quadratic Casimirs, mean-square order of convergence one, 
weak order of convergence two. These properties are illustrated with numerical experiments. 
\end{abstract}


\maketitle
{\small\noindent 
{\bf AMS Classification (2020).} 65C30, 65P10, 60H10, 60H35.}

Stochastic Hamiltonian systems. Stochastic Poisson systems. Energy. Casimir. 
Trace formula. Numerical schemes. 
Strong convergence. Weak convergence.

\section{Introduction}

Hamiltonian systems are widely used models in science and engineering. 
In the deterministic case, one main feature of such models is that the solution conserves exactly the Hamiltonian energy 
for all times. The design and study of energy-preserving numerical methods for such 
problems as attracted much attention in the recent years, see for instance \cite{bgi18,MR2926250,cos14,ch11,g96,h10,it09,k16,m14,mb16,mqr99,qm08,wws13} and references therein.

For an additive white noise perturbation of such Hamiltonian systems, 
the energy is no longer constant along time, but grows in average linearly for the exact solution, 
which reveals non trivial to reproduce by numerical methods, 
see \cite{bb14,c12,cjz17,cs12,hsw06,smh04,st15,ccal19}, 
and extensions to the case of stochastic partial differential equations 
in \cite{ac18,aclw16,cchs19,cls13,h08}. 

In this paper, we propose and analyze a drift-preserving scheme for  
stochastic Poisson systems subject to an additive noise. 
Such problems are a direct generalization of the stochastic differential equations (SDEs) studied recently 
in \cite{ccal19}, as well as in all the above references, but the proposed 
numerical integrator is not a trivial generalization of the one given in \cite{ccal19}. 

In Section~\ref{sect-dp} we propose a new numerical method that exactly satisfies a  
trace formula for linear growth for for all times of the expected value of the Hamiltonian energy and of the Casimir of the solution. Such long time behavior corresponds to the one 
of the exact solution of stochastic Poisson systems and can also be seen 
as a long time weak convergence estimate. 
For the sake of completeness, in Section~\ref{sect-conv}, we prove mean-square 
and weak orders of convergence of the proposed numerical method under classical assumptions on the coefficients of the problem.  
Finally, Section~\ref{sec:numerics} is devoted to numerical experiments illustrating the main properties 
of the new numerical method for stochastic Hamiltonian systems and Poisson systems. 

\section{Drift-preserving scheme for stochastic Poisson problem}\label{sect-dp}
This section presents the problem, introduces the drift-preserving integrator 
and shows some of its main geometric properties.

\subsection{Setting}
For a fixed dimension $d$, let $W(t) \in\IR^d$ denote a standard $d$-dimensional Wiener process defined for $t>0$ on a probability space equipped with a filtration and fulfilling the usual assumptions.
For a fixed dimension $m$ and a smooth potential $V\colon \R^m\to\R$, 
we consider the separable Hamiltonian function of the form
\begin{align}\label{sepHam}
H(p,q)=\frac12\sum_{j=1}^mp_j^2+V(q).
\end{align}
We next set $X(t)=(p(t),q(t))\in\IR^m\times\IR^m$ and consider the following stochastic Poisson system with additive noise
\begin{equation}\label{prob}
\text dX(t)=B(X(t))\nabla H(X(t))\,\text dt+\begin{pmatrix}\Sigma \\ 0\end{pmatrix}\,\text dW(t).
\end{equation}
Here, $B(X)\in\R^{2m\times 2m}$ is a smooth skew-symmetric matrix and $\Sigma\in\R^{m\times d}$ is a constant matrix. 
The initial value $X_0=(p_0,q_0)$ of the problem \eqref{prob} is assumed to be either non-random 
or a random variable with bounded moments up to any order (and adapted to the filtration). 
For simplicity, we assume in the analysis of this paper that $(x,y) \mapsto B(x)\nabla H(y)$ 
is globally Lipschitz continuous on $\IR^{2m}\times \IR^{2m}$ and 
that $H$ and $B$ are $C^7$, resp. $C^6$-functions with all partial derivatives 
with at most polynomial growth. This is to ensure existence and uniqueness 
of solutions to \eqref{prob} for all times $t>0$ 
as well as bounded moments at any orders of such solutions. 
These regularity assumptions on the coefficients $B$ and $H$ will also be used 
to show strong and weak convergence of the proposed 
numerical scheme for \eqref{prob}. We observe that one could weaken these assumptions, 
but this is not the aim of the present work.
The present setting covers, for instance, the following examples: 
simplified versions of the stochastic rigid bodies studied in \cite{tshy19,MR2644322}, 
the stochastic Hamiltonian systems 
considered in \cite{ccal19} by taking 
$$
B(X)=J=\begin{pmatrix} 0& -Id_m\\ Id_m& 0\end{pmatrix}
$$ 
the constant canonical symplectic matrix, for which the SDE \eqref{prob} yields
$$
\text dp(t) = - \nabla V(q(t))\,\text dt + \Sigma\,\text dW(t), \qquad \text dq(t) = p(t)\,\text dt,
$$
the Hamiltonian considered in \cite{bb14} (where the matrix $\Sigma$ is diagonal), 
the linear stochastic oscillator from \cite{smh04}, and 
various stochastic Hamiltonian systems studied in \cite[Chap. 4]{Milstein2004}, see also \cite{MR1897950}, 
or \cite{Seesselberg1994,MR3658901,MR3882980,MR3873562}.

\begin{remark}\label{rem1}
We emphasize that our analysis is not restricted to the above form of the Hamiltonian. 
Indeed, the results below as well as the proposed numerical scheme 
can be applied to the more general problem (no needed of partitioning the vector $X$ neither to have 
the separable Hamiltonian \eqref{sepHam})
$$
\text dX(t)=B(X(t))\nabla H(X(t))\,\text dt+\begin{pmatrix}\Sigma \\ 0\end{pmatrix}\,\text dW(t),
$$
as long as the Hessian of the Hamiltonian has a nice structure. One could for instance consider a (linear in $p$) term of 
the form $\tilde V(q)p$ or most importantly the case when the Hamiltonian is quadratic 
as in the example of a stochastic rigid body problem. See below for further details. 
\end{remark}

Applying It\^o's lemma to the function $H(X)$ on the solution process $X(t)$ of \eqref{prob}, one obtains
\begin{align}\label{lasuperformule}
\text dH(X(t))&=\left(\nabla H(X(t))^\top B(X(t))\nabla H(X(t))
+\frac12\Tr\left( \begin{pmatrix}\Sigma\\0\end{pmatrix}^\top
\nabla^2H\begin{pmatrix}\Sigma\\0\end{pmatrix}\right)  \right)\,\text dt \nonumber\\
&\quad+\nabla H(X(t))^\top\begin{pmatrix}\Sigma\\0\end{pmatrix}\,\text dW(t).
\end{align}
Using the skew-symmetry of the matrix $B(X)$, we have $\nabla H(X)^TB(X)\nabla H(X)=0$.  
Furthermore, using that the partial Hessian $\nabla_{pp}^2H(X)=Id_m$ is a constant matrix, 
thanks to the separable form of the Hamiltonian \eqref{sepHam}, and rewriting the above equation 
in integral form and taking the expectation, one finally obtains the so-called 
\emph{trace formula for the energy} of the stochastic Poisson SDE \eqref{prob}:
\begin{equation}\label{etrace}
\E\left[H(X(t))\right]=\E\left[H(X_0)\right]+\frac12\Tr\left(\Sigma^\top\Sigma\right)t.
\end{equation}
This shows that the expected energy of the exact solution of \eqref{prob} 
grows linearly with time for all $t>0$.
\begin{remark}
Observe that the form of the noise term in equation~\eqref{prob} makes the term  
$$
\Tr\left(\begin{pmatrix}\Sigma\\0\end{pmatrix}^\top
\nabla^2H\begin{pmatrix}\Sigma\\0\end{pmatrix}\right)=\Tr\left(\Sigma^\top\Sigma\right)
$$
in \eqref{lasuperformule} independent of the stochastic process $X(t)$. 
Hence one obtains the linear growth along time of the expected energy in \eqref{etrace}. 
In general, this is not the case if one would consider a non-zero additive noise in all the component 
or a multiplicative noise in \eqref{prob}. Note however that the linear growth property of the expected energy  
is still valid if one considers a more general Hamiltonian function \eqref{sepHam} with kinetic energy given by 
$\frac12p^\top M^{-1}p$, with a given invertible mass matrix $M$.
\end{remark}

Our objective is to derive and analyze a new numerical scheme for \eqref{prob} that possesses the same 
trace formula for the energy for all times. 

\subsection{Definition of the numerical scheme}
The numerical integrator studied in \cite{ccal19} cannot directly be applied to 
the stochastic Poisson system \eqref{prob}. Our idea is to combine 
a splitting scheme with one of the (deterministic) energy-preserving schemes from \cite{ch11}. 
Observe that a similar strategy was independently presented in \cite{chs19} in the particular 
context of the Langevin equation with other aims than here. We thus propose the following time integrator for problem~\eqref{prob}, 
which is shown in Theorem~\ref{thmTrace} below to be a \emph{drift-preserving integrator} for all times: 
\begin{equation}\label{dp}
\begin{split}
Y_1&:=X_n+\begin{pmatrix}\Sigma\\0\end{pmatrix}\left(W(t_n+\frac{h}2)-W(t_n)\right),\\
Y_2&:=Y_1+hB\left(\frac{Y_1+Y_2}2\right)\int_0^1\nabla H(Y_1+\theta(Y_2-Y_1))\,\text d\theta,\\
X_{n+1}&=Y_2+\begin{pmatrix}\Sigma\\0\end{pmatrix}\left(W(t_{n+1})-W(t_n+\frac{h}2)\right).
\end{split}
\end{equation}
In the above formulas, we denote the step size of the drift-preserving scheme with $h>0$ and discrete times with $t_n=nh$. 

\begin{remark}\textbf{(Numerical implementation)}\label{rem-implem}
The middle step in equation~\eqref{dp} requires, in general, the solution to a nonlinear system of equations. 
Even in higher dimension, if the problem is not stiff, this can be solved by fixed point iterations rather than Newton iterations, which makes the computational complexity similar to that of an implicit Runge--Kutta scheme with two stages, see \cite[Section~2.2]{ch11} 
or \cite[Chapter VIII.6]{MR2249159} for instance.
\end{remark}

\begin{remark}\textbf{(Further extensions)} 
Let us observe that the (deterministic) energy-preserving scheme from \cite{ch11} present in 
the term in the middle of \eqref{dp} could be replaced by another (deterministic) energy-preserving scheme 
for (deterministic) Poisson systems, see for example: \cite{MR2926250,MR2926249,MR3787423,MR3993177} 
or a straightforward adaptation of the energy-preserving Runge--Kutta schemes for polynomial Hamiltonians  
in \cite{MR2542869}. Let us further remark that it is also possible to interchange the ordering 
in the splitting scheme 
by considering first half a step of the (deterministic) energy-preserving scheme, then a full step of 
the stochastic part, and finally again half a step of the (deterministic) energy-preserving scheme. 
Finally, let us add that one could add a damping term in the SDE \eqref{prob} to compensate for the drift in 
the energy thus getting conservation of energy for such problems (either in average or a.s.). 
In this case, one would add the damping term 
in the first and last equations of the numerical scheme \eqref{dp} in order to get a (stochastic) 
energy-preserving splitting scheme. An example of application is Langevin's equation, 
a widely studied model in the context of molecular dynamics. We do not pursue further this question.
\end{remark}

We now show the boundedness along time of all moments of the numerical solution given by \eqref{dp}.

\begin{lemma}\label{lbm}
Let $T>0$. Apply the drift-preserving numerical scheme \eqref{dp} to the Poisson system with additive 
noise \eqref{prob} on the compact time interval $[0,T]$. 
One then has the following bounds for the numerical moments: 
for all step sizes $h$ assumed small enough and all $m\in\N$,
$$
\E[|X_n|^{2m}] \leq C_m, 
$$
for all $t_n=nh \leq T$, where $C_m$ is independent of $n$ and $h$. 
\end{lemma}
\begin{proof}
To show boundedness of the moments of the numerical solution given by \eqref{dp}, we use \cite[Lemma 2.2, p. 102]{Milstein2004}, 
which states that it is sufficient to show 
the following estimates:
$$
\left|\E\left[X_{n+1}-X_n|X_n\right]\right|\leq C\left(1 + |X_n|\right)h\quad \text{and}\quad
\left|X_{n+1}-X_n\right|\leq M_n (1 + |X_n|)\sqrt{h},
$$
where $C$ is independent of $h$ and $M_n$ is a random variable with moments of all orders bounded
uniformly with respect to all $h$ small enough. Since the numerical scheme~\eqref{dp} is a splitting method, it is 
more convenient to apply \cite[Lemma 2.2, p. 102]{Milstein2004} to the Markov chain $\{X_0,Y_1,Y_2,X_1,\ldots\}$ instead of 
the Markov chain $\{X_0,X_1,\ldots\}$. This makes the verification of the above estimates immediate using the linear growth property  
of the coefficients of the SDE \eqref{prob}, a consequence of their Lipschitzness. 
\end{proof}

\subsection{Exact drift preservation of energy}
We are now in position to prove the main feature of the proposed numerical method \eqref{dp} which benefits from the very same trace formula for the energy as the one 
for the exact solution to the stochastic Poisson problem \eqref{prob}, 
hence the name drift-preserving integrator for this numerical scheme.  
\begin{theorem}\label{thmTrace}
Consider the numerical scheme \eqref{dp} applied to the Poisson system with additive noise \eqref{prob}. 
Then, for all time steps $h$ assumed small enough, 
the expected energy of the numerical solution satisfies the following trace formula
\begin{equation}\label{ntrace}
\E\left[H(X_n)\right]=\E\left[H(X_0)\right]+\frac12\Tr\left(\Sigma^\top\Sigma\right)t_n
\end{equation}
for all discrete times $t_n=nh$, where $n\in\mathbb N$. 
\end{theorem}
\begin{proof}
The first step of the drift-preserving scheme can be rewritten as 
$$
Y_1=X_n+\int_{t_n}^{t_n+\frac{h}2}\begin{pmatrix}\Sigma\\0\end{pmatrix}\,\text dW(s)
$$
and an application of It\^o's formula gives
$$
\E\left[H(Y_1)\right]=\E\left[H(X_n)\right]+\frac{h}4\Tr\left(\Sigma^\top\Sigma\right).
$$
Since the second step of the drift-preserving scheme \eqref{dp} is 
the deterministic energy-preserving scheme from \cite{ch11}, one then obtains
$$
\E\left[H(Y_2)\right]=\E\left[H(Y_1)\right].
$$
Finally, as in the beginning of the proof, the last step of the numerical integrator provides 
\begin{align*}
\E\left[H(X_{n+1})\right]&=\E\left[H(Y_2)\right]+\frac{h}4\Tr\left(\Sigma^\top\Sigma\right)=
\E\left[H(Y_1)\right]+\frac{h}4\Tr\left(\Sigma^\top\Sigma\right)\\
&=\E\left[H(X_n)\right]+\frac{h}2\Tr\left(\Sigma^\top\Sigma\right).
\end{align*}
The identity \eqref{ntrace} then follows by induction on $n$.
A recursion now completes the proof.
\end{proof}

\subsection{Splitting methods with deterministic symplectic integrators and backward error analysis: linear case}\label{sec-bea}
As symplectic integrators for deterministic Hamiltonian systems or Poisson integrators for deterministic Poisson systems 
have proven to be very successful \cite[Chapters VI and VII]{haluwa}, 
it may be tempting to use them in a splitting scheme for the SDE \eqref{prob}. 
One could for instance replace the (deterministic) energy-preserving scheme in the middle step 
of equation \eqref{dp} by a symplectic or Poisson integrator, such as for instance the second order 
St\"ormer--Verlet method \cite[Sect.\ts 5]{MR2249159} which turns out to be explicit in the context of a separable Hamiltonian \eqref{sepHam}. 
Using a backward error analysis, 
see \cite[Chapter 10]{MR1270017}, \cite[Chapter IX]{haluwa}, \cite[Chapter 5]{MR2132573}, or \cite[Chapter 5]{MR3642447}, 
one arrives at the following result in the case of a linear Hamiltonian system with additive noise \eqref{prob} 
(i.\,e.\ts for a quadratic potential $V$), where the proposed splitting scheme is drift-preserving for a modified Hamiltonian.

\begin{proposition} \label{proposition:bea}
For a quadratic potential $V$ in \eqref{sepHam}, consider the numerical discretization of the Hamiltonian system with additive noise \eqref{prob} (where $B(x)=J$ for ease of presentation) by 
the drift-preserving numerical scheme \eqref{dp}, where 
the energy-preserving scheme in the middle $Y_1 \mapsto Y_2$ is replaced by a deterministic symplectic partitioned Runge--Kutta method of order $p$. 
Then, there exists a modified Hamiltonian $\widetilde H_h$ which is a 
quadratic perturbation of size $\bigo{h^p}$ of the original Hamiltonian $H$, 
such that 
the expected energy satisfies the following trace formula for all time steps $h$ assumed small enough, 
\begin{equation}\label{ntracebea}
\E\left[\widetilde H_h(X_n)\right]=\E\left[\widetilde H_h(X_0)\right]+\frac12\Tr\left(\Sigma^\top \widetilde\sigma_h\Sigma\right)t_n, 
\end{equation}
for all discrete times $t_n=nh$, where $n\in\mathbb N$,
and $\widetilde\sigma_h=\nabla^2_{pp} \widetilde H_h(x)$ is a constant matrix (independent of $x$).
\end{proposition}
\begin{proof}
By backward error analysis and the theory of modified equations, see for instance \cite[Chapter IX]{haluwa}, 
the symplectic Runge--Kutta method $Y_1\mapsto Y_2$ solves exactly a modified Hamiltonian system with 
initial condition $Y_1$ and modified Hamiltonian $\widetilde H_h(x)=H(x) + \bigo{h^p}$ given by a formal series which turns out to be convergent in the linear case for all $h$ small enough (and with a quadratic modified Hamiltonian).
Following the lines of the proof of Theorem~\ref{thmTrace} applied with the modified Hamiltonian $\widetilde H_h$, 
and observing that the partial Hessian $\nabla^2_{pp} \widetilde H_h(x)$ is a constant matrix independent of $x$ (as $\widetilde H_h$ is quadratic), we deduce the estimate \eqref{ntracebea} for the averaged modified energy.
\end{proof}
Observe in \eqref{ntracebea} that the constant scalar 
$\frac12\Tr\left(\Sigma^\top \widetilde\sigma_h \Sigma\right)=\frac12\Tr\left(\Sigma^\top \Sigma\right)+\bigo{h^p}$ 
is independent of $x$ and a perturbation of size $\bigo{h^p}$ of the drift rate for the exact solution of the SDE 
in~\eqref{ntrace}. 

Finally, note that an analogous result in the nonlinear setting (with nonquadratic potential $V$ in~\eqref{sepHam}) does not seem straightforward due in particular to the non-boundedness of the moments of the numerical solution 
over long times and the fact that the modified Hamiltonian $\widetilde H_h(p,q)$ is nonquadratic with respect to $p$ in general for a nonquadratic potential $V$.

\subsection{Exact drift preservation of quadratic Casimir's}
We now consider the case where 
the ordinary differential equation (ODE) coming from \eqref{prob}, 
i.\,e. equation \eqref{prob} with $\Sigma=0$,  
has a quadratic Casimir of the form 
$$
C(X)=\frac12X^\top AX,
$$
with a symmetric constant matrix 
$$
A=\begin{pmatrix}a & b\\b^\top & c\end{pmatrix}
$$
with $a,b,c\in\mathbb{R}^{m\times m}$. 
Recall that a function $C(X)$ is called a Casimir if $\nabla C(X)^\top B(X)=0$ for all $X$. 
Along solutions to the ODE, we thus have $C(X(t))=\text{Const}$. This property is independent 
of the Hamiltonian $H(X)$.

In this situation, one can show a \emph{trace formula for the Casimir} 
as well as a drift-preservation of this 
Casimir for the numerical integrator \eqref{dp}. 
\begin{proposition}\label{propCasimir}
Consider the numerical discretization 
of the Poisson system with additive noise \eqref{prob} with the Casimir $C(X)$ 
by the drift-preserving numerical scheme \eqref{dp}. 
Then, 
\begin{enumerate}
\item the exact solution to the SDE \eqref{prob} has the following trace formula for the Casimir
\begin{equation}\label{eCas}
\E\left[C(X(t))\right]=\E\left[C(X_0)\right]+\frac{a}2\Tr\left(\Sigma^\top \Sigma\right)t,
\end{equation}
for all times $t>0$. 
\item the numerical solution \eqref{dp} has the same trace formula for the Casimir, for all time steps $h$ assumed small enough,
\begin{equation}\label{nCas}
\E\left[C(X_n)\right]=\E\left[C(X_0)\right]+\frac{a}2\Tr\left(\Sigma^\top \Sigma\right)t_n,
\end{equation}
for all discrete times $t_n=nh$, where $n\in\mathbb N$. 
\end{enumerate}
\end{proposition}
\begin{proof}
The above results can be obtain directly by applying It\^o's 
formula and using the property of the Casimir function $C(X)$.
\end{proof}
Stochastic models with such a quadratic Casimir naturally appear for a simplified version of a stochastic rigid body motion of a spacecraft 
from \cite{tshy19} which has the quadratic Casimir 
$C(X)=\norm{X}^2_2$ or a reduced model for the rigid body in a solvent from \cite{MR2644322}. 
See also the numerical experiments in Section \ref{sec:RB}. 

\section{Convergence analysis}\label{sect-conv}
In this section, we study the mean-square and weak convergence of the drift-preserving 
scheme \eqref{dp} on compact time intervals under the classical setting of globally Lipschitz continuous coefficients.

\subsection{Mean-square convergence analysis}\label{sect-ms}
In this subsection, we show the mean-square convergence of the drift-preserving 
scheme \eqref{dp} on compact time intervals under the classical setting of 
Milstein's fundamental theorem \cite[Theorem 1.1]{Milstein2004}. 
\begin{theorem}\label{th-ms}
Let $T>0$. Consider the Poisson problem with additive noise \eqref{prob} 
and the drift-preserving integrator \eqref{dp}. 
Then, for all time steps $h$ assumed small enough, 
the numerical scheme \eqref{dp} converges with order $1$ in the mean-square sense:
$$
\left(\E[\norm{X(t_n)-X_n}^2]\right)^{1/2}\leq Ch,
$$
for all $t_n=nh\leq T$, where the constant $C$ is independent of $h$ and $n$.
\end{theorem} 
\begin{proof} 
Denoting $f(x)=B(x)\nabla H(x)$, a Taylor expansion of $f$ in the exact solution of \eqref{prob} gives 
\begin{align*}
X(h)
&=
X_0+hf(X_0)+\begin{pmatrix}\Sigma\\0\end{pmatrix}W(h)
+
hf'(X_0)\begin{pmatrix}\Sigma\\0\end{pmatrix}\int_0^h W(t)\,\dd t
+REST_1,
\end{align*}
where the term (denoting $f''$ the bilinear form for the second order derivative of $f$)
$$
REST_1=f'(X_0)\int_0^h\int_0^tf(X(s))\,\text{d}s+ \int_0^h\int_0^1 (1-\theta)f''(X_0+\theta(X(t)-X_0) )
\left(X(t)-X_0,X(t)-X_0\right)
\dd \theta\,\dd t
$$ 
is bounded in the mean and mean-square sense as follows: 
\begin{equation}\label{eqtictac}
\E[REST_1]\leq Ch^2 \quad\text{and}\quad \E[\norm{REST_1}^2]^{1/2}\leq Ch^{2},
\end{equation}
where $C$ is a constant independent of $h$, but that depends on $X_0=x$ with at most a polynomial growth.
Performing a Taylor expansion of $f$ in the numerical solution \eqref{dp} gives, after some straightforward computations,
\begin{align*}
X_1&=X_0+hf(X_0)+\begin{pmatrix}\Sigma\\0\end{pmatrix}W(h)
+hf'(X_0)\begin{pmatrix}\Sigma\\0\end{pmatrix}W\left(\frac h2\right)
+REST_2,
\end{align*}
where the term $REST_2$ analogously satisfies the bounds \eqref{eqtictac}.

The above computations result in the following local error estimates,
\begin{equation}\label{eq:localerr}
\E[X(h)-X_1]=\bigo{h^2},\qquad \E[\norm{X(h)-X_1}^2]^{1/2}=\bigo{h^{3/2}},
\end{equation}
where the constants in $\mathcal{O}$ depend on $X_0=x$ with at most a polynomial growth.
An application of Milstein's fundamental theorem, see \cite[Theorem 1.1]{Milstein2004}, 
finally shows that the scheme \eqref{dp} converges with global order of convergence $1$ in the mean-square sense, as consequence of the local error estimates \eqref{eq:localerr} and Lemma~\ref{lbm}. This concludes the proof.
\end{proof}

\subsection{Weak convergence analysis}\label{sect-weak}
The proof of weak convergence of the drift-preserving 
scheme \eqref{dp} on compact time intervals easily follows from 
\cite[Proposition~6.1]{MR3246903}, where convergence of the Strang splitting scheme 
for SDEs is shown. See also \cite{MR3735294,MR3927434} for related results.
\begin{theorem}\label{th-we}
Let $T>0$. Consider the Poisson problem with additive noise \eqref{prob} 
and the drift-preserving integrator \eqref{dp}. 
Then, there exists $h^*>0$ such that for all $0<h\leq h^*$, 
the numerical scheme converges with order $2$ in the weak sense: for all $\Phi\in C_P^6(\R^{2m},\R)$, the space of $C^6$ functions with all derivatives up to order $6$ with at most polynomial growth, one has
$$
\left|\E[\Phi(X(t_n))]-\E[\Phi(X_n)]\right|\leq Ch^2,
$$
for all $t_n=nh\leq T$, where the constant $C$ is independent of $h$ and $n$.
\end{theorem} 

\section{Numerical experiments}\label{sec:numerics}
In this section, we illustrate numerically the above analysis of the 
proposed drift-preserving scheme \eqref{dp}, denoted by \textsc{DP} below. 
Furthermore, we compare it with the well known integrators, 
in particular the Euler--Maruyama scheme (denoted by \textsc{EM}) 
and the backward Euler--Maruyama scheme (denoted by \textsc{BEM}). 
The first and second Hamiltonian test problems considered (linear
stochastic oscillator and stochastic mathematical pendulum) use parameter values 
similar to those in \cite{ccal19}. The third test problem is a stochastic rigid body model 
which is a Poisson system perturbed by white noise, but not a Hamiltonian system. 
For nonlinear problems, we use fixed-point iterations for the implementation of the schemes, but one could 
use Newton iterations as well, see Remark~\ref{rem-implem}. 

\subsection{The linear stochastic oscillator} \label{sec:linear}
The linear stochastic oscillator has extensively been used as a test model since the seminal work \cite{smh04}.  
We thus first consider the SDE \eqref{prob} with $B(X)=J$ the constant $2\times2$ Poisson matrix and 
the following Hamiltonian
$$
H(p,q)=\frac12p^2+\frac12q^2.
$$
Furthermore, the initial values are $(p_0,q_0)=(0,1)$ and we consider a one dimensional noise with parameter $\Sigma=1$.  

For this problem, the integral present in the drift-preserving scheme~\eqref{dp} 
can be computed exactly, resulting in an explicit time integrator: 
\begin{equation*}
\begin{split}
Y_1&:=X_n+\begin{pmatrix}\left(W(t_n+\frac{h}2)-W(t_n)\right)\\0\end{pmatrix},\\
Y_2&:=\frac{1}{1+\frac{h^2}4}\begin{pmatrix}1-\frac{h^2}4 & -h\\ h & 1-\frac{h^2}4 \end{pmatrix}Y_1,\\
X_{n+1}&=Y_2+\begin{pmatrix}\left(W(t_{n+1})-W(t_n+\frac{h}2)\right)\\0\end{pmatrix}.
\end{split}
\end{equation*}
This numerical scheme is different from the one proposed in \cite{ccal19}.

In Figure \ref{fig:traceSO}, we compute the expected values of the energy $H(p,q)$ for various numerical integrators. 
This is done using the step sizes $h=5/2^4$, resp. $h=100/2^8$, and the time intervals $[0,5]$, resp. $[0,100]$. 
For the numerical discretization of the linear stochastic oscillator, we choose 
the (backward) Euler--Maruyama schemes (\textsc{EM} and \textsc{BEM}), 
the drift-preserving scheme (\textsc{DP}), 
and also the stochastic trigonometric method from \cite{c12} (\textsc{STM}). 
For the considered problem, the stochastic trigonometric method (\textsc{STM}) also has an exact trace formula for the energy, 
see \cite{c12} for details. We approximate the values of the expected energies using 
averages over $M=10^6$ samples. 
Similarly to the stochastic trigonometric method (\textsc{STM}) from \cite{c12}, one can observe the perfect long time behavior of the drift-preserving scheme with exact averaged energy drift along time, as stated 
in Theorem~\ref{thmTrace}. In contrast, the left picture of Figure \ref{fig:traceSO} illustrates that the expected energy of the classical Euler--Maruyama scheme drifts exponentially with time, while  the backward Euler--Maruyama scheme exhibits an inaccurately slow growth rate, as emphasized in \cite{smh04}. 

\begin{figure}[tb]
\centering
\includegraphics*[height=5cm,keepaspectratio]{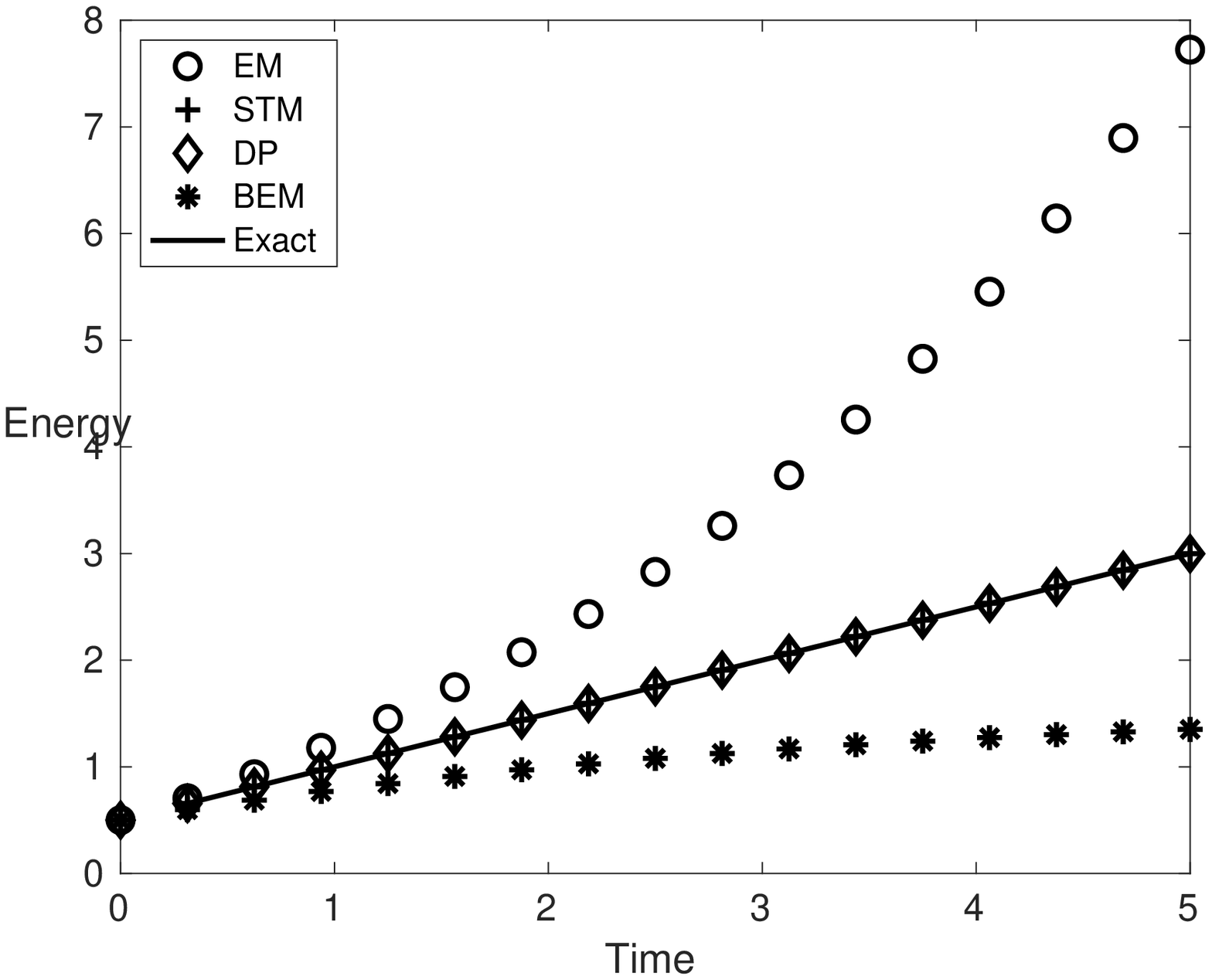}
\qquad
\includegraphics*[height=5cm,keepaspectratio]{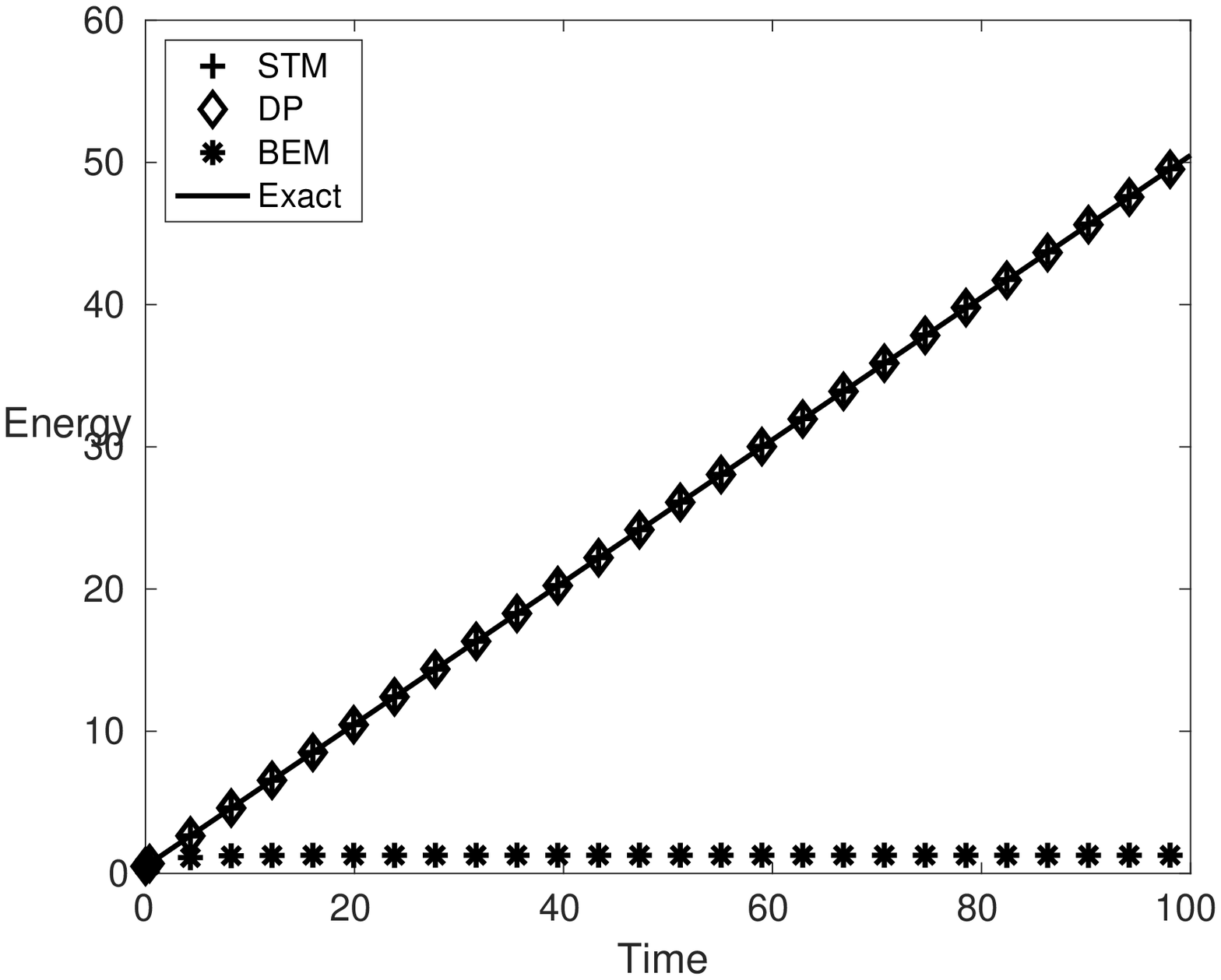}
\caption{Linear stochastic oscillator: numerical trace formulas for 
$\IE[H(p(t),q(t))]$ on the interval $[0,5]$ (left) and $[0,100]$ (right).
Comparison of the Euler--Maruyama scheme (EM), the stochastic trigonometric method (STM), the drift-preserving scheme (DP), the backward Euler--Maruyama scheme (BEM), and the exact solution.}
\label{fig:traceSO}
\end{figure}

In Figure~\ref{fig:msSO}, we illustrate numerically the strong rate of convergence of the drift-preserving scheme \eqref{dp} and compare with the other schemes. 
To this aim, we discretize the linear stochastic oscillator on the time interval $[0,1]$ 
using step sizes ranging from $h=2^{-6}$ to $h=2^{-10}$ and we use as a reference solution the stochastic trigonometric method with small time step $h_{\text{ref}}=2^{-12}$. The expected values are approximated by computing averages over $M=10^6$ samples. One can observe the mean-square order $1$ of convergence of the drift-preserving scheme \eqref{dp} with lines of slope 1 in Figure \ref{fig:msSO}, which corroborates Theorem~\ref{th-ms}.  

\begin{figure}[h]
\centering
\includegraphics*[height=5cm,keepaspectratio]{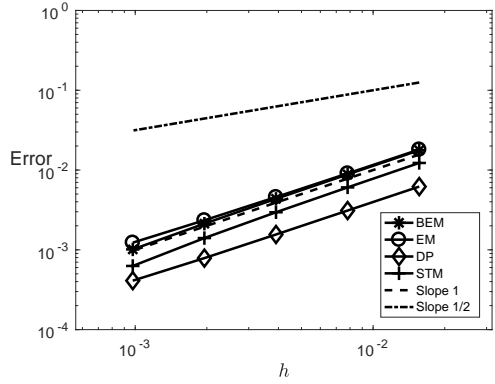}
\caption{Linear stochastic oscillator: mean-square convergence rates for the backward Euler--Maruyama scheme (BEM), the Euler--Maruyama scheme (EM), the drift-preserving scheme (DP), 
and the stochastic trigonometric method (STM). Reference lines of slopes 1, resp. 1/2.}
\label{fig:msSO}
\end{figure}

Next, Figure~\ref{fig:weakSO} illustrates the weak convergence rate of the drift-preserving scheme \eqref{dp}. For simplicity, we only display the errors in the first and second moments 
since explicit formulas are available for these quantities. 
We take the noise scaling parameter $\Sigma=0.1$ and step sizes ranging from $h=2^{-4}$ to $h=2^{-16}$. 
The remaining parameters are the same as in the previous numerical experiment.
The lines of slope $2$ in  Figure~\ref{fig:weakSO} illustrates that the drift-preserving scheme has a weak order of convergence $2$ 
in the first and second moments, as stated in Theorem~\ref{th-we}. 

\begin{figure}[tb]
\centering
\begin{subfigure}[b]{1\textwidth}
\centering
   \includegraphics*[height=4.0cm,keepaspectratio]{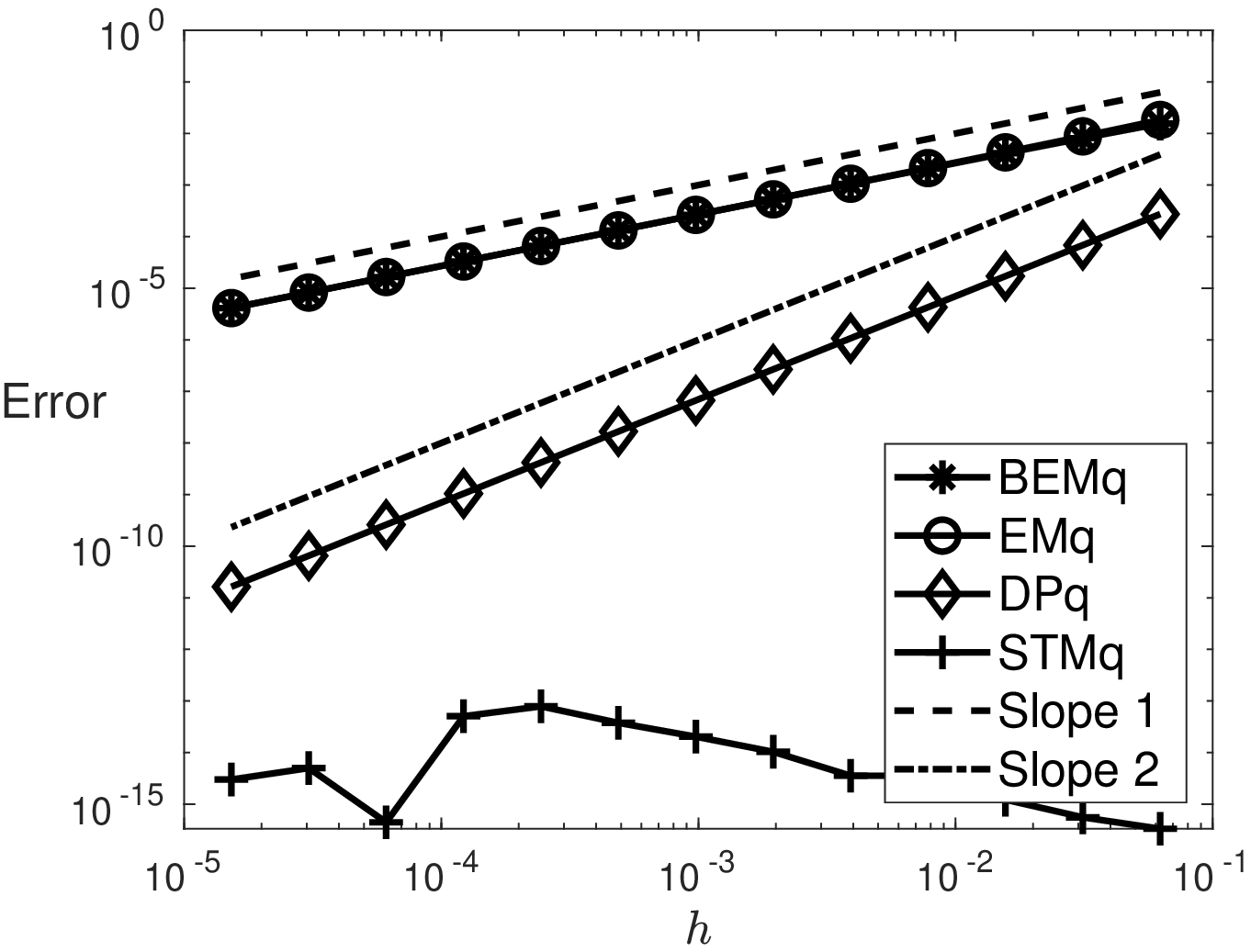}
	 \quad
   \includegraphics*[height=4.0cm,keepaspectratio]{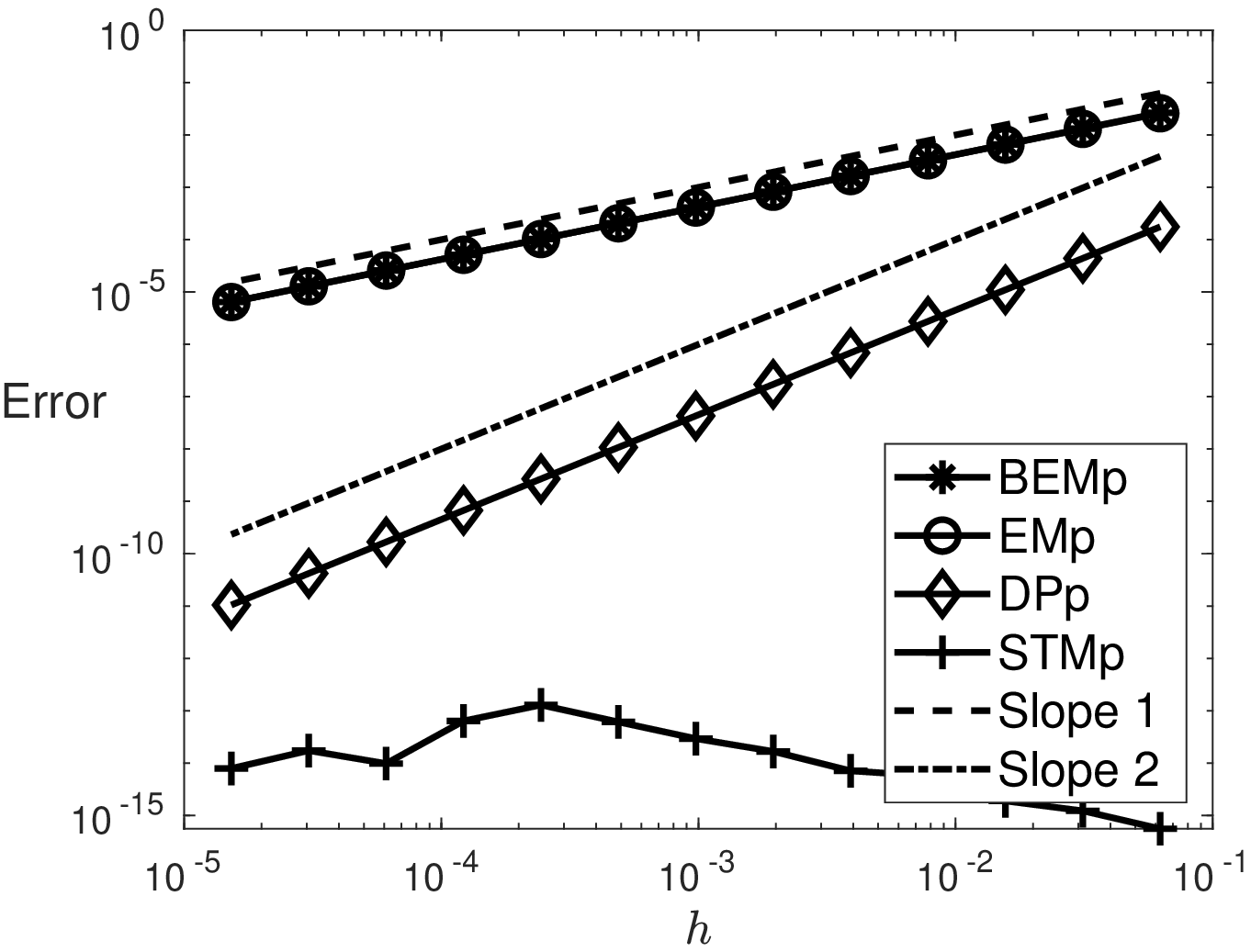}
   \caption{Errors in the first moments $\IE[q(t)]$ (left) and $\IE[p(t)]$ (right).}
   \label{fig:w1} 
\end{subfigure}
\begin{subfigure}[b]{1\textwidth}
\centering
   \includegraphics*[height=4.0cm,keepaspectratio]{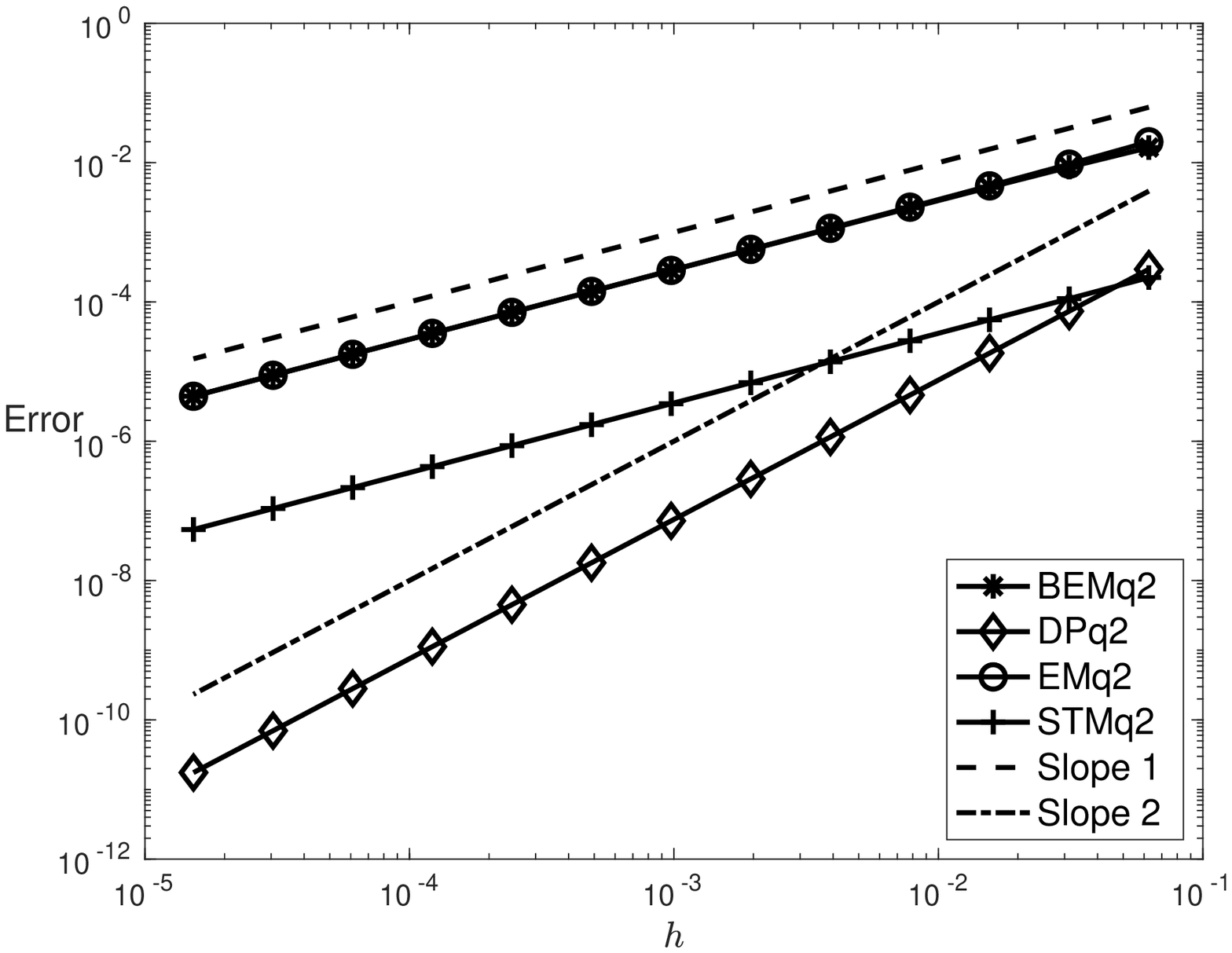}
  \quad 
	\includegraphics*[height=4.0cm,keepaspectratio]{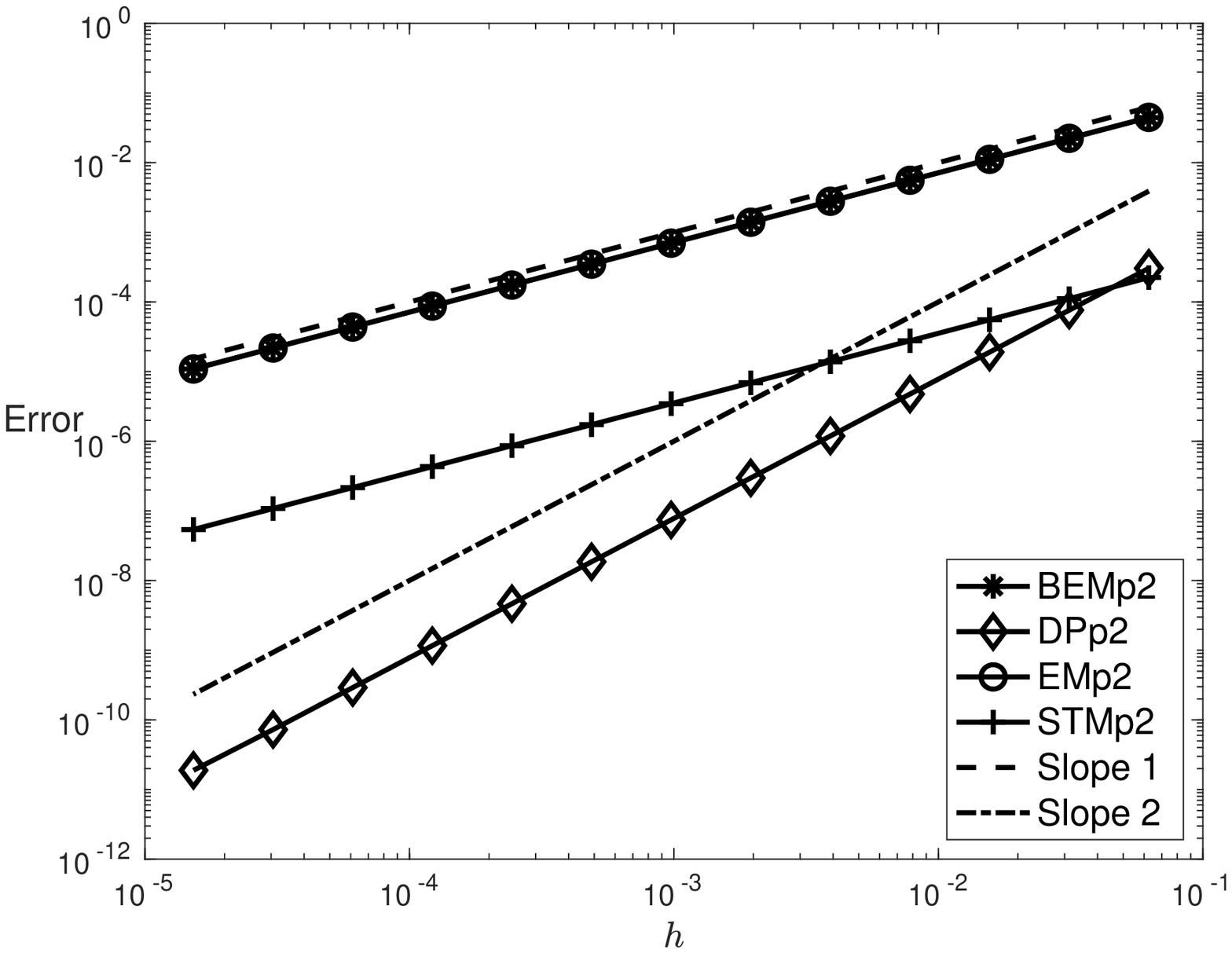}
   \caption{Errors in the second moments $\IE[q(t)^2]$ (left) and $\IE[p(t)^2]$ (right).}
   \label{fig:w2} 
\end{subfigure}
\caption{Linear stochastic oscillator: weak convergence rates 
for the backward Euler--Maruyama scheme (BEM), the Euler--Maruyama scheme (EM), the drift-preserving scheme (DP), 
and the stochastic trigonometric method (STM). Reference lines of slopes 1, resp. 2.}
\label{fig:weakSO}
\end{figure}

As symplectic integrators for deterministic Hamiltonian systems have proven to be very successful \cite{haluwa}, 
it may be tempting to use them in a splitting scheme for the SDE \eqref{prob}. 
To study this, in Figure~\ref{fig:traceOscVS}, we compare the behavior, with respect to the trace formula, 
of the drift-preserving scheme and of the symplectic splitting strategies discussed in Section~\ref{sec-bea}. We use 
the classical Euler symplectic and St\"ormer--Verlet schemes for the deterministic Hamiltonian 
and integrate the noisy part exactly. These numerical integrators are denoted by SYMP, resp. ST below. 
As a comparison with non-geometric numerical integrators, we also use the classical Euler and Heun's schemes 
in place of a symplectic scheme. These numerical integrators are denoted by splitEULER and splitHEUN. 
We discretize the linear stochastic oscillator on the time interval $[0,100]$ with $2^7$ step sizes. It can be observed that the splitting method using the non-symplectic schemes splitEULER or splitHEUN
behaves as poorly as standard explicit schemes for SDEs: we hence display in Figure~\ref{fig:traceOscVS} only part of
their numerical values due to their exponential growth. Although not having the exact growth rates, 
the two symplectic splitting schemes behave much better than the classical Euler--Maruyama scheme 
with a linear drift in the averaged energy with a perturbed rate, as predicted by Proposition~\ref{proposition:bea}.

\begin{figure}[tb]
\centering
\includegraphics*[height=5cm,keepaspectratio]{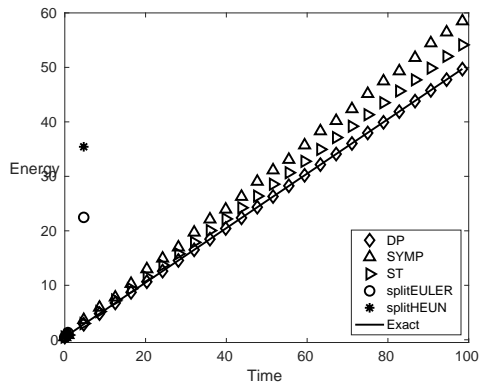}
\caption{Linear stochastic oscillator: numerical trace formulas for $\IE[H(p(t),q(t))]$ on the interval $[0,100]$.
Comparison of the drift-preserving scheme (DP), the splitting methods with, respectively, the symplectic Euler method (SYMP), the  
St\"ormer–Verlet method
(ST), the explicit Euler method (splitEULER), the Heun method (splitHEUN), and the exact solution.
}
\label{fig:traceOscVS}
\end{figure}

\subsection{The stochastic mathematical pendulum}
Let us next consider the nonlinear SDE~\eqref{prob} 
(with $B(X)=J$ the constant canonical Poisson matrix) with the Hamiltonian
$$
H(p,q)=\frac12p^2-\cos(q)
$$
and a noise in dimension one with parameter $\Sigma=1$. We take the initial values $(p_0,q_0)=(1,\sqrt{2})$. 

We again compare the behavior, with respect to the trace formula, 
of the DP, SYMP, ST and splitEULER schemes. 
To do this, we integrate numerically the stochastic mathematical pendulum on the time interval $[0,100]$ with $2^7$ step sizes.
The results are presented in Figure~\ref{fig:tracePendVS}. 
Again, we recover the fact that the drift-preserving scheme exhibits 
the exact averaged energy drift, as predicted in Theorem~\ref{thmTrace}. 
Furthermore, one can still observe a good behavior of the symplectic strategies 
from Section~\ref{sec-bea} analogously to the linear case in Section~\ref{sec:linear}, 
although the analysis in Proposition~\ref{proposition:bea} is only valid for the linear case. 

\begin{figure}[h]
\centering
\includegraphics*[height=5cm,keepaspectratio]{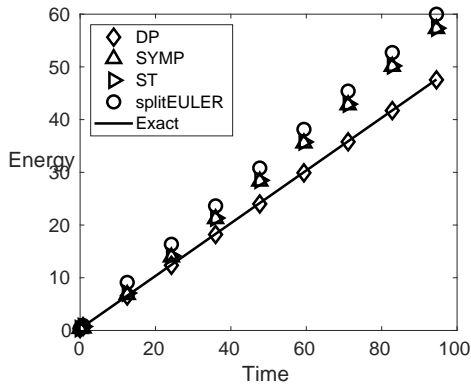}
\caption{Stochastic mathematical pendulum: numerical trace formulas for $\IE[H(p(t),q(t))]$ on the interval $[0,100]$.
Comparison of the drift-preserving scheme (DP), the splitting methods with, respectively, the symplectic Euler method (SYMP), the  
St\"ormer–Verlet method
(ST), the explicit Euler method (splitEULER), and the exact solution.}
\label{fig:tracePendVS}
\end{figure}

\subsection{Stochastic rigid body problem} \label{sec:RB}
We now consider an It\^o version of the stochastic rigid body problem 
studied in \cite{MR1432623,acvz,cd14} for instance. 
This system has the following Hamiltonian
$$
H(X)=\frac12\left(X_1^2/I_1+X_2^2/I_2+X_3^2/I_3\right), 
$$
the quadratic Casimir
$$
C(X)=\frac12\left(X_1^2+X_2^2+X_3^2\right), 
$$
and the skew-symmetric matrix 
$$
B(X)=\begin{pmatrix} 0 & -X_3 & X_2 \\ X_3 & 0 & -X_1 \\ -X_2 & X_1 &0\end{pmatrix}.
$$
Here, we denote the angular momentum by $X=(X_1,X_2,X_3)^\top$ and take the moments of inertia to be 
$I=(I_1,I_2,I_3)=(0.345, 0.653, 1)$. The initial value for the SDE \eqref{prob}
is given by $X(0)=(0.8, 0.6, 0)$ and we consider a scalar noise $W(t)$ with $\Sigma=0.25$ 
(acting on the first component $X_1$ only). 

Observe that, even if the Hamiltonian has not the desired structure \eqref{sepHam}, 
one still has a linear drift in the energy since the Hamiltonian is quadratic and 
thus the Hessian matrix present in the derivation of the trace formula 
has the correct structure as noted in Remark~\ref{rem1}.

In Figure~\ref{fig:traceRB}, we compute the expected values of the energy $H(X)$ and the Casimir $C(X)$ 
using $N=2^5$ step sizes on the time interval $[0,4]$ (in order to 
see also the behavior of the Euler--Maruyama scheme) along various numerical solutions. 
The expected values are approximated by computing averages over $M=10^6$ samples.
The exact long time behavior with respect to the energy and the Casimir averaged growths 
of the drift-preserving scheme can be observed in Figure~\ref{fig:traceRB}, which corroborates Theorem~\ref{thmTrace} and Proposition~\ref{propCasimir}. 
As in the previous numerical experiment, one can also see that the growth rates 
of the Euler--Maruyama schemes are in contrast qualitatively 
wrong.

\begin{figure}[tb]
\centering
\includegraphics*[height=5cm,keepaspectratio]{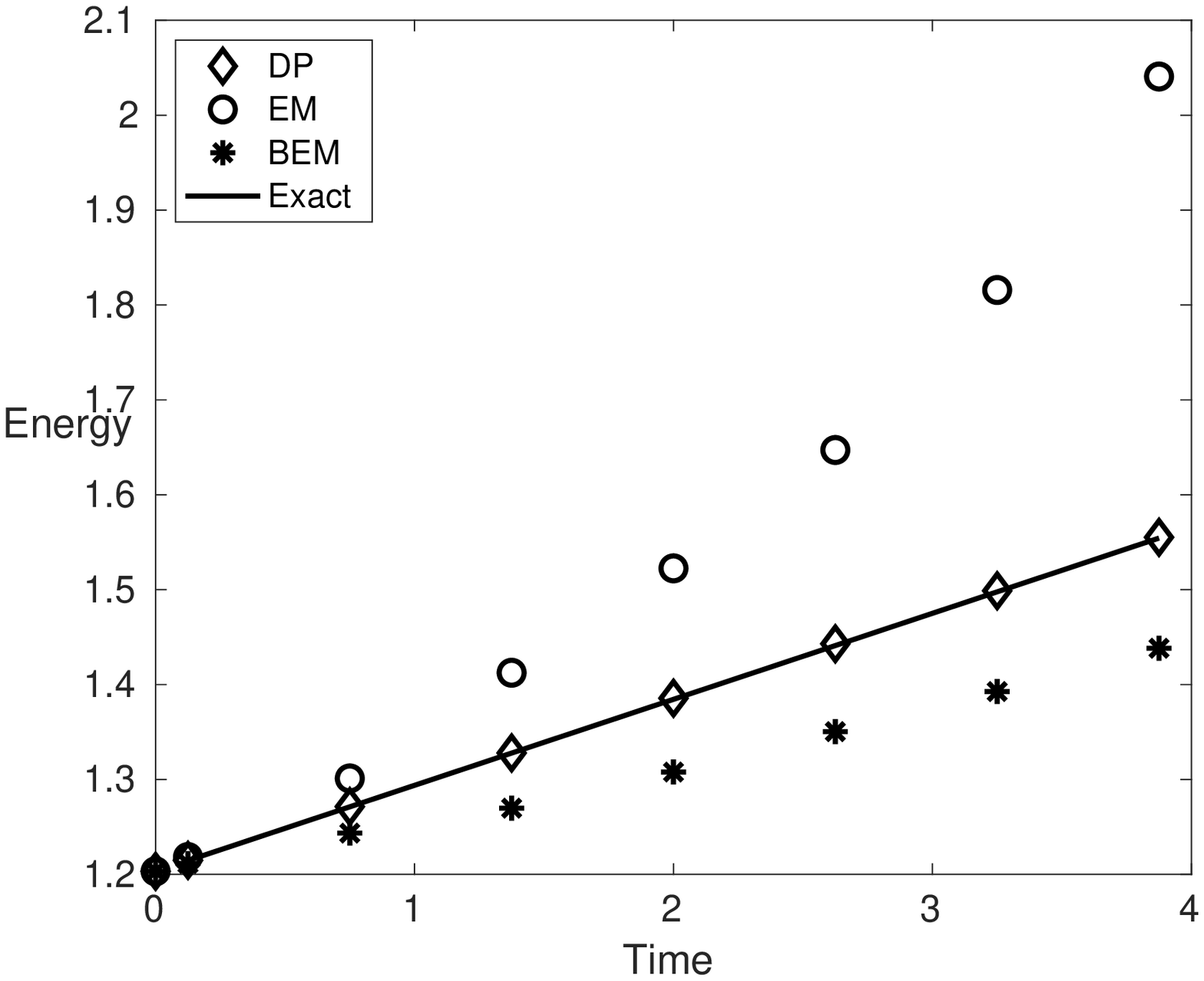}
\qquad \includegraphics*[height=5cm,keepaspectratio]{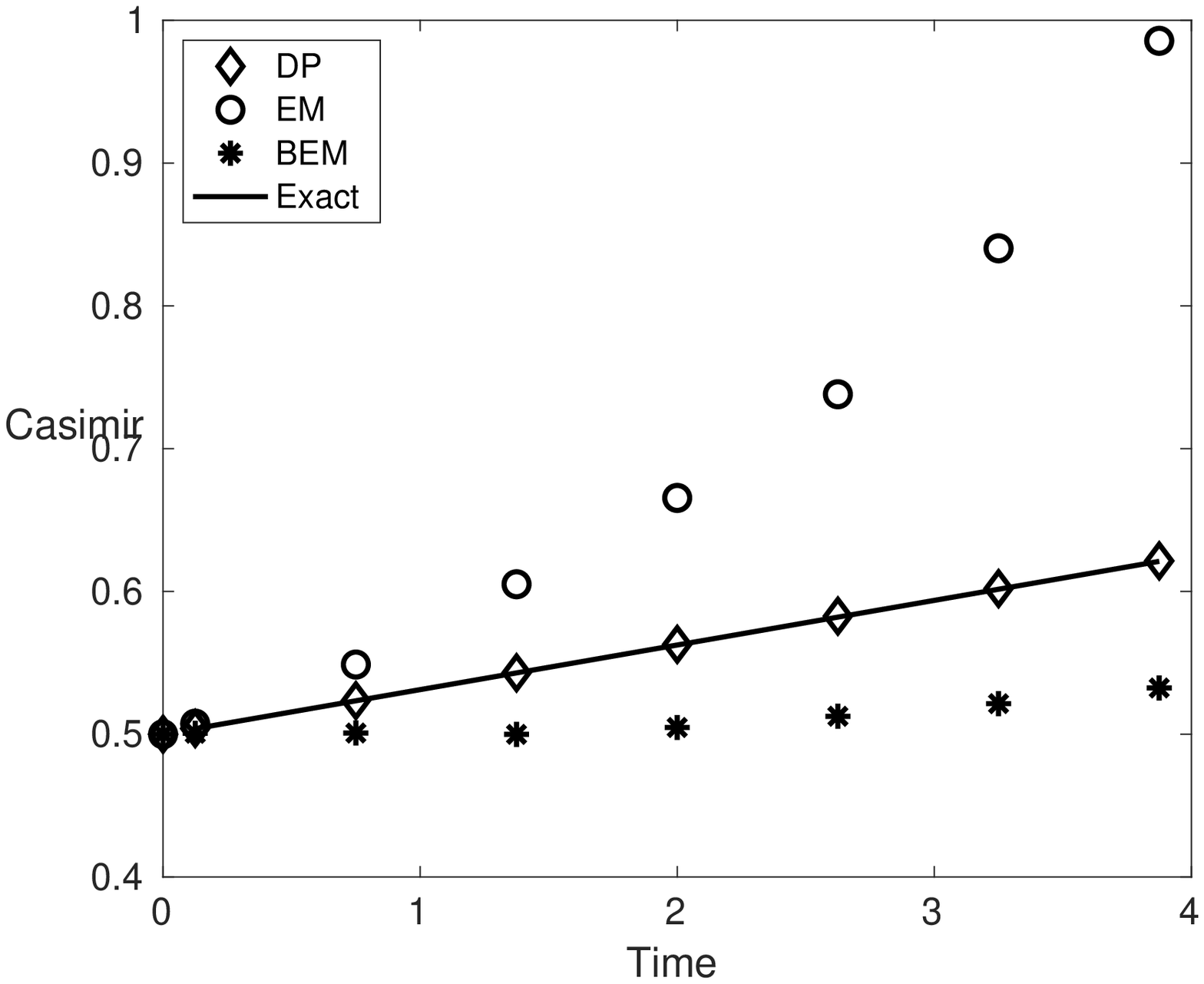}
\caption{Stochastic rigid body problem: numerical trace formulas 
for the energy $\IE[H(X(t))]$ (left) and for the Casimir $\IE[C(X(t))]$ (right) for the drift-preserving scheme (DP), 
the Euler--Maruyama scheme (EM), the backward Euler--Maruyama scheme (BEM), and the exact solution.}
\label{fig:traceRB}
\end{figure}

Similarly to the previous example, we numerically illustrate in Figure~\ref{fig:msRB} the strong convergence rate of the 
drift-preservation scheme \eqref{dp} for the stochastic rigid body problem. 
To this aim, we discretize the problem on the time interval $[0,0.75]$ 
using step sizes ranging from $h=2^{-6}$ to $h=2^{-10}$ and compare with a reference solution obtained with scheme \eqref{dp} with $h_{\text{ref}}=2^{-12}$. We compute averages over $M=10^5$ samples to approximate the expected values present in the mean-square errors. One observes again mean-square convergence of order $1$ for the drift-preserving scheme. 

\begin{figure}[tb]
\centering
\includegraphics*[height=5cm,keepaspectratio]{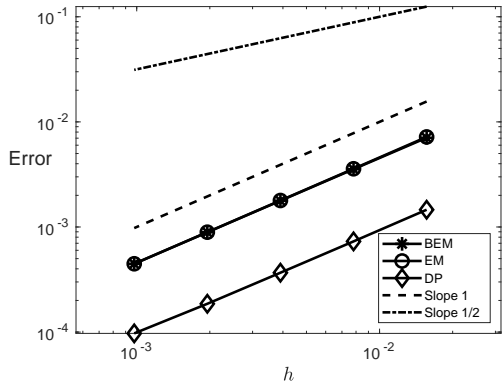}
\caption{Stochastic rigid body problem: mean-square convergence rates
 for the backward Euler--Maruyama scheme (BEM), 
the drift-preserving scheme (DP), 
and the Euler--Maruyama scheme (EM).
Reference lines of slopes 1, resp. 1/2.}
\label{fig:msRB}
\end{figure}

Next, Figure~\ref{fig:weakRB} illustrates the weak convergence rate of the drift-preserving scheme \eqref{dp}. We plot the weak errors in the first and second moments of the first component 
of the solutions using the parameters: $\Sigma=0.1$, $T=1$, $M=2500$ samples, 
and step sizes ranging from $h=2^{-10}$ to $h=2^{-20}$. The rest of the parameters are as in the previous numerical experiment. 
One can observe weak order $2$ in the first and second moments for the drift-preserving scheme (up to Monte-Carlo errors), which confirms again the statement of Theorem~\ref{th-we}.

\begin{figure}[tb]
\centering
\includegraphics*[height=4.6cm,keepaspectratio]{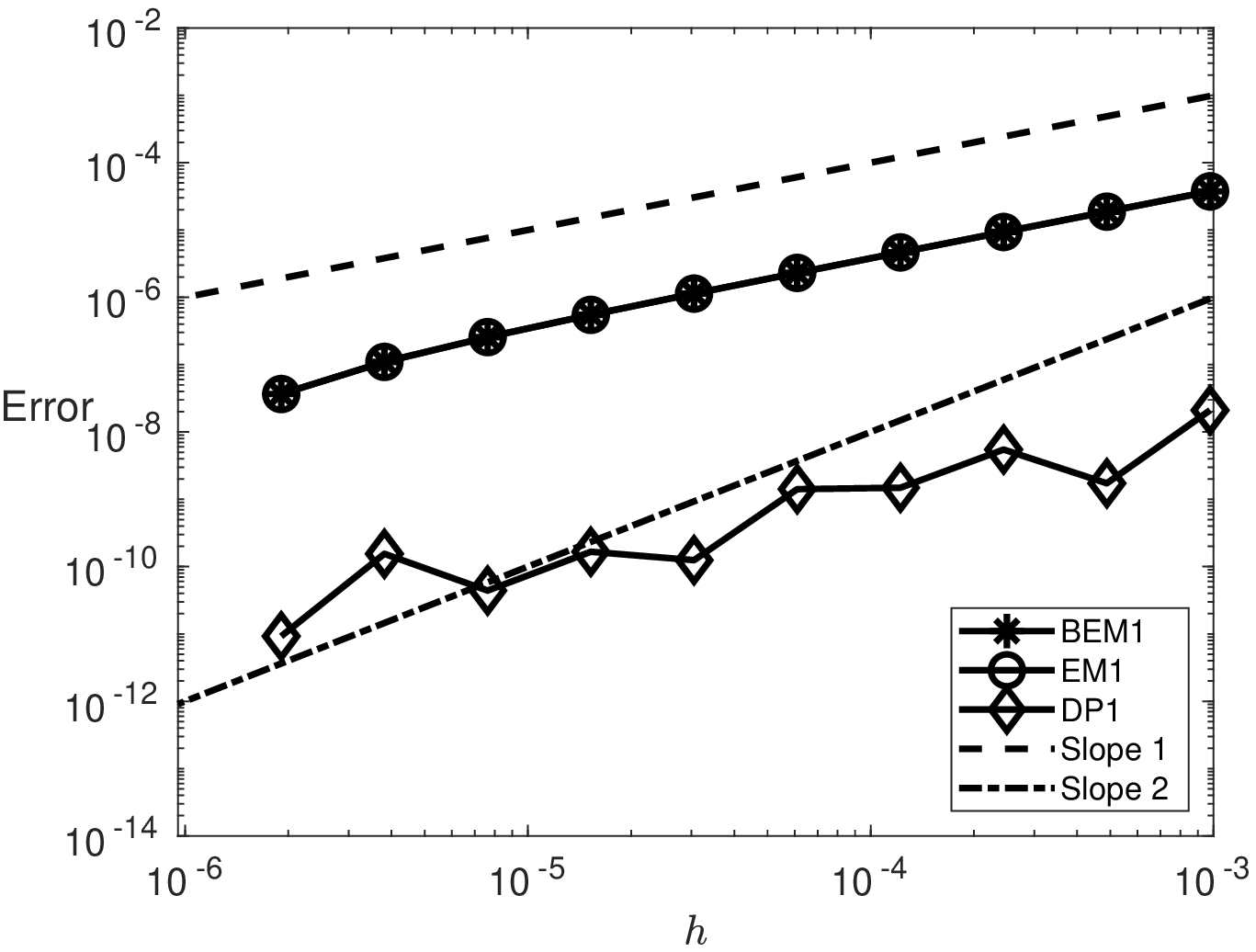}
\quad
\includegraphics*[height=4.5cm,keepaspectratio]{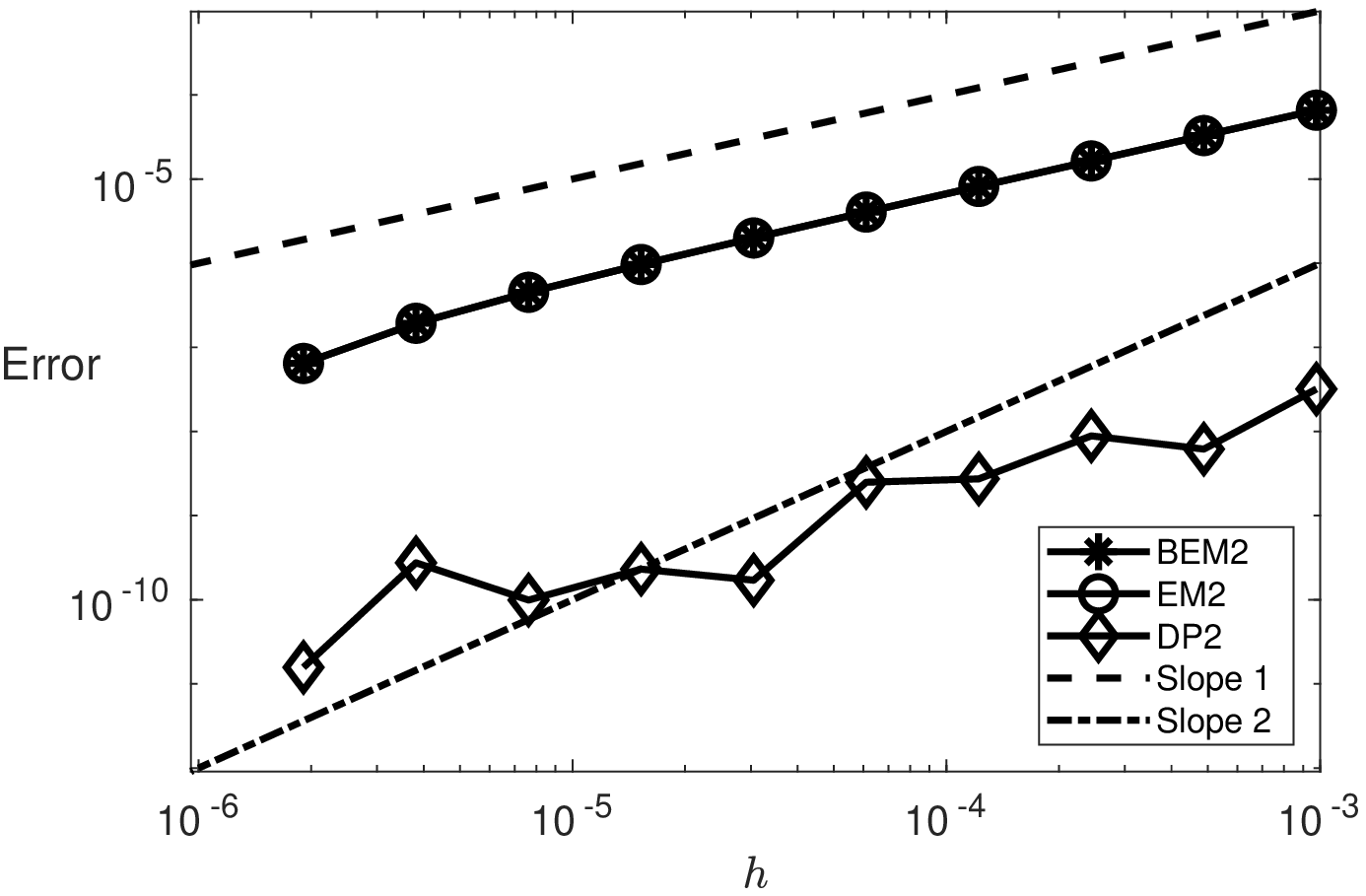}
\caption{Stochastic rigid body problem: weak convergence rates in the first moment $\IE[X_1(t_n)]$ (left) and second moment $\IE[X_1(t_n)^2]$ (right) for the drift-preserving scheme (DP), 
the Euler--Maruyama scheme (EM), and the backward Euler--Maruyama scheme (BEM). Reference lines of slopes 1, resp. 2.}
\label{fig:weakRB}
\end{figure}

Finally, in Figure~\ref{fig:revtraceRB}, we take the same parameters as in the first experiment of this subsection but we consider a noise in dimension two with the matrix 
$$
\Sigma=
\begin{pmatrix}
0.25 & 0\\ 0 &0.25
\end{pmatrix}.
$$
We then compute the expected values of the energy $H(X)$ and the Casimir $C(X)$ 
using $N=2^6$ step sizes along various numerical solutions and 
display the trace formula for the energy
$$
\E\left[H(X(t))\right]=\E\left[H(X_0)\right]+\frac12\Tr\left(\Sigma^\top\begin{pmatrix}1/I_1 & 0\\ 0 & 1/I_2\end{pmatrix}\Sigma\right)t 
$$
and the trace formula for the Casimir
$$
\E\left[C(X(t))\right]=\E\left[C(X_0)\right]+\frac{1}2\Tr\left(\Sigma^\top \Sigma\right)t.
$$
\begin{figure}[tb]
\centering
\includegraphics*[height=5cm,keepaspectratio]{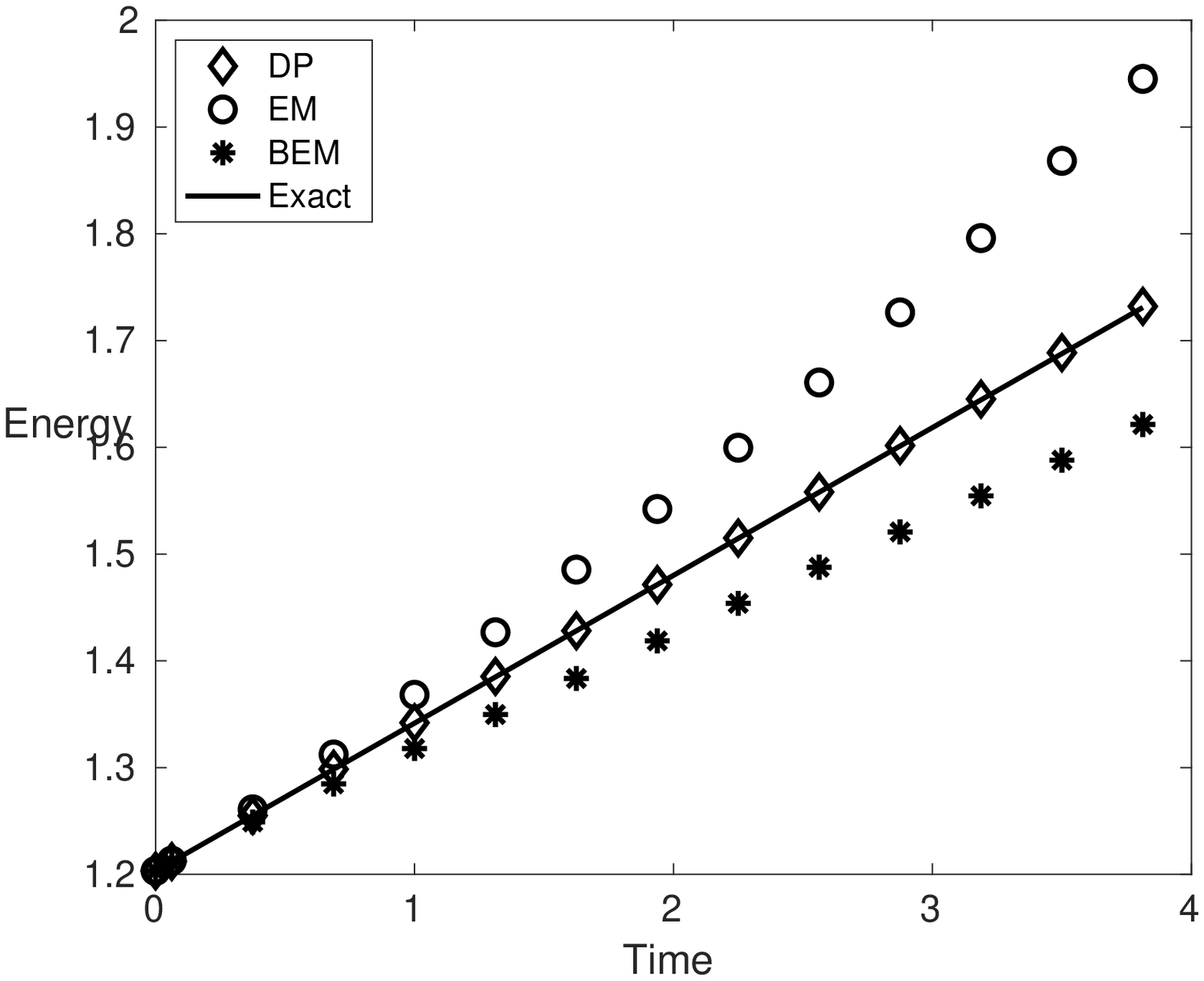}
\qquad \includegraphics*[height=5cm,keepaspectratio]{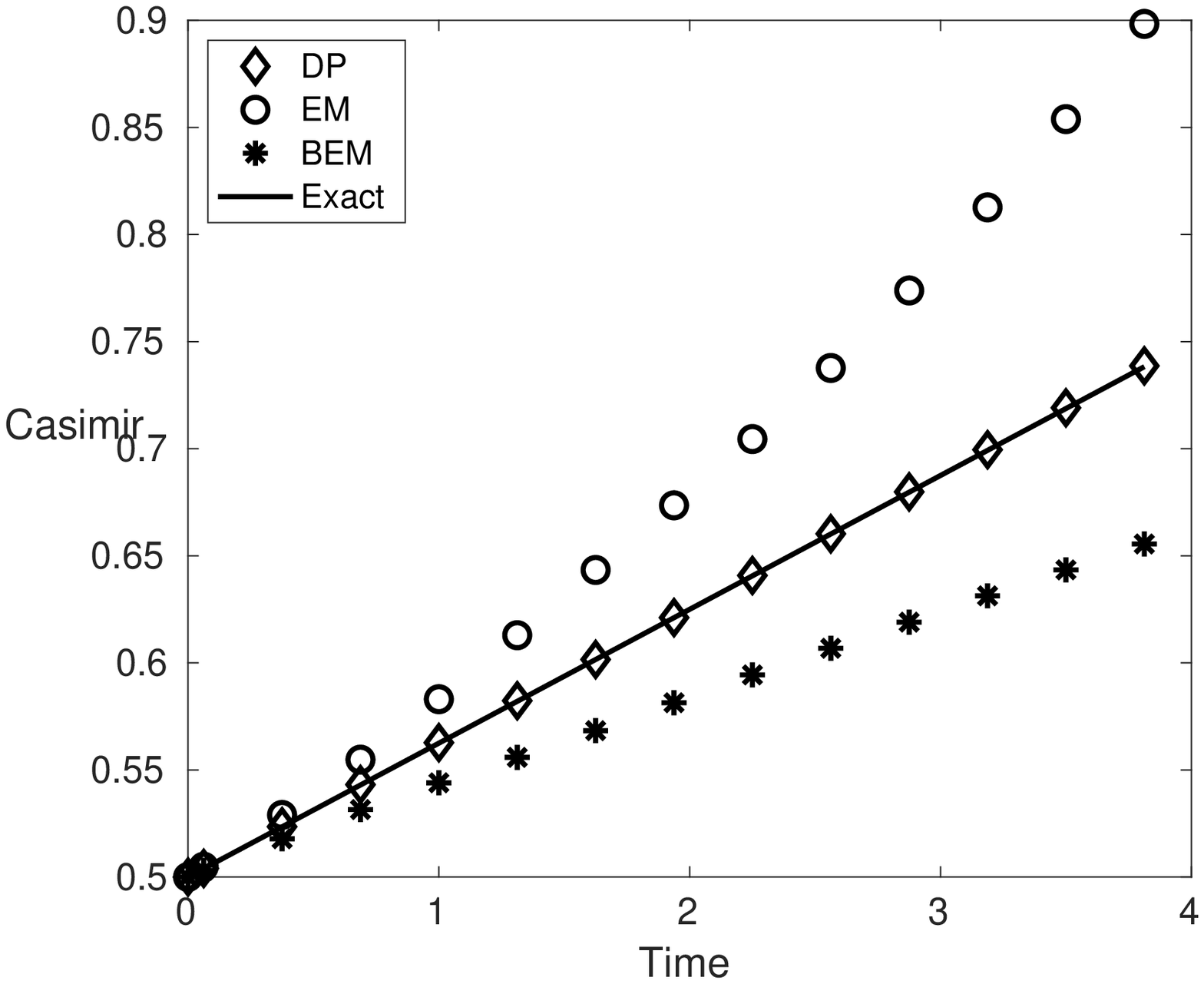}
\caption{Stochastic rigid body problem with two dimensional noise: numerical trace formulas for the energy $\IE[H(X(t))]$ (left) and for the Casimir $\IE[C(X(t))]$ (right) for the Casimir $\IE[C(X(t))]$ (right) for the drift-preserving scheme (DP), 
the Euler--Maruyama scheme (EM), the backward Euler--Maruyama scheme (BEM), and the exact solution.}
\label{fig:revtraceRB}
\end{figure}
Again, one can observe in Figure~\ref{fig:revtraceRB} the excellent behavior of the drift-preserving scheme as stated 
in Theorem~\ref{thmTrace} and Proposition~\ref{propCasimir}. 

\section{Acknowledgements}
We appreciate the referees' comments on an earlier version of the paper. 
The work of DC was supported by the Swedish Research Council (VR) (projects nr. $2018-04443$). 
The work of GV was partially supported by the Swiss National Science Foundation, grants No. 200020\_184614, No. 200021\_162404 and No. 200020\_178752. The computations were performed on resources provided by the Swedish National Infrastructure for Computing (SNIC) at HPC2N, Ume{\aa} University and UPPMAX, Uppsala University. 

\bibliographystyle{plain}
\bibliography{biblio}

\end{document}